\documentclass[12pt, reqno]{amsproc}
\usepackage[utf8]{inputenc}
\usepackage[margin=1.1in]{geometry}
\usepackage{comment,mathrsfs}
\usepackage{hyperref}
\hypersetup
{
    colorlinks,	%
    citecolor=red,%
    filecolor=black,%
    linkcolor=blue,%
    urlcolor=blue
}
\usepackage{mathpazo} 
\usepackage{enumitem}
\usepackage{multicol}
\renewcommand{\phi}{\varphi}

\renewcommand{\mathcal}{\mathscr}
\usepackage[all,2cell,cmtip]{xy}
\UseAllTwocells

\usepackage{amsthm,amsmath, amssymb}
\usepackage{mathtools,stmaryrd}
\usepackage{blkarray}

\usepackage[ruled,vlined]{algorithm2e}
\SetKw{Break}{break}

\usepackage{tikz} 
\usetikzlibrary{calc,decorations.markings,matrix,arrows,cd,math,decorations.pathmorphing}
\usepackage{graphicx}
\graphicspath{ {./images/} }
 \tikzset{>=latex}
 \usepackage{multirow,subcaption}
 
\newcommand{\set}[1]{\left\{{#1}\right\}}

\newcommand{\R}{\mathbb R}
\newcommand{\Z}{\mathbb Z}
\newcommand{\N}{\mathbf N}
\newcommand{\T}{\mathbf T}
\newcommand{\C}{\mathbb C}

\newcommand{\K}{\mathbb K}
\newcommand{\I}{\mathbf I}
\newcommand{\id}{\mathrm{id}}

\newcommand{\cA}{\mathcal{A}}
\newcommand{\cL}{\mathcal{L}}

\newcommand\ie{\textit{i.e.} }
\DeclareMathOperator{\Imag}{Im}
\DeclareMathOperator{\Real}{Re}

\DeclareMathOperator{\Vect}{\text{\bf Vect}}

\DeclareMathOperator{\Rep}{\text{\bf Rep}}
\DeclareMathOperator{\PRep}{\text{\bf PRep}}

\DeclareMathOperator{\Hom}{Hom}
\DeclareMathOperator{\barc}{\text{\bf Bar}}
\DeclareMathOperator{\bS}{\mathscr{S}}

\newcommand{\udim}{{\underline\dim}}
\newcommand{\UQ}{\text{\bf U}Q}
\newcommand{\bu}{\bar{u}}
\newcommand{\bv}{\bar{v}}

 \newcommand*\quot[2]{
        \mathchoice
            {
                \text{\raise1ex\hbox{$#1$}\Big/\lower1ex\hbox{$#2$}}%
            }
            {
                #1\,\big/\,#2
            }
            {
                #1\,/\,#2
            }
            {
                #1\,/\,#2
            }
    }
    
\newtheorem{thm}{Theorem}[section]
\newtheorem{cor}[thm]{Corollary}
\newtheorem{lmm}[thm]{Lemma}
\newtheorem{prop}[thm]{Proposition}
\newtheorem*{thm*}{Theorem}

\theoremstyle{definition}
\newtheorem{defi}[thm]{Definition}

\newtheorem{ex}[thm]{Example}
\newtheorem{rmk}[thm]{Remark}

\title{Harder-Narasimhan Filtrations and Zigzag Persistence}
\author{Marc Fersztand, Vidit Nanda and Ulrike Tillmann}

\begin{document}

\begin{abstract}
    We introduce a sheaf-theoretic stability condition for finite acyclic quivers. Our main result establishes that for representations of affine type $\widetilde{\mathbb{A}}$ quivers, there is a precise relationship between the associated Harder-Narasimhan filtration and the barcode of the periodic zigzag persistence module obtained by unwinding the underlying quiver.
\end{abstract}
\maketitle

\section*{Introduction}

This paper concerns representations of acyclic quivers of affine type $\widetilde{\mathbb{A}}_n$. The underlying graph of any such quiver is the $n$-cycle as drawn below, but one does not obtain a directed cycle after the edges have been oriented:
\[
\xymatrixcolsep{.5in}
\xymatrixrowsep{.8in}
\xymatrix{
& & x_{n-1} \ar@{-}[dll]_-{e_0} \ar@{-}[drr]^-{e_{n-1}} & & \\
x_0 \ar@{-}[r]_-{e_1} & x_1 \ar@{-}[r]_-{e_2} & \cdots  \ar@{-}[r]_-{e_{n-3}} & x_{n-3} \ar@{-}[r]_-{e_{n-2}} & x_{n-2}
}
\]

Our goal here is to describe remarkable formulas which relate two discrete quantities that are associated to every finite-dimensional representation $V$ of such a quiver. The first one has its roots in geometric invariant theory, and constitutes a numerical reduction of $V$'s Harder-Narasimhan filtration along a special choice of stability condition. The second quantity of interest arises in the algebraic study of persistent homology. To obtain it, one first lifts $V$ to an $n$-periodic zigzag persistence module over the infinite quiver
\[
\xymatrixcolsep{.4in}
\xymatrix{
\cdots \ar@{-}[r]^-{e_{n-1}} & x_{n-1} \ar@{-}[r]^-{e_{0}} & x_0 \ar@{-}[r]^-{e_{1}}  & x_1 \ar@{-}[r]^-{e_{2}} & \cdots  \ar@{-}[r]^-{e_{n-2}} &  x_{n-2} \ar@{-}[r]^-{e_{n-1}} &  x_{n-1} \ar@{-}[r]^-{e_{0}} & \cdots,
}
\]
and then catalogues the multiplicities of its indecomposable summands. Before outlining the main contributions of our work, we provide brief summaries of both quantities below.

\subsection*{Harder-Narasimhan Filtrations} For the purposes of these introductory remarks, a {\em stability condition} on a finite acyclic quiver $Q$ with vertex set $Q_0$ is a map $\alpha:Q_0 \to \mathbb{R}$ that assigns a real number $\alpha_x$ to each vertex $x$. The $\alpha$-{\em slope} of a nonzero finite-dimensional representation $V$ of $Q$ is the ratio
\[
\phi_\alpha(V) := \frac{\sum_{x \in Q_0} \alpha_x \cdot \dim V_x}{\sum_{x \in Q_0} \dim V_x}.
\]
Here $V_x$ denotes the vector space assigned by $V$ to each vertex $x$. We call $V$ {\em semistable} if the inequality $\phi_\alpha(U) \leq \phi_\alpha(V)$ holds for every nonzero subrepresentation $U \subset V$. Once we fix $\alpha$, there exists a unique, finite length filtration $V^\bullet$ of $V$:
\[
0= V^0 \subsetneq V^1 \subsetneq \cdots \subsetneq V^\ell = V,
\]
where the successive quotient representations $S^j := V^j/V^{j-1}$ are semistable and have strictly decreasing $\alpha$-slopes. This $V^\bullet$ is called the Harder-Narasimhan filtration \cite{Harder1974, Reineke_2003} of $V$ along $\alpha$. Our first invariant, for a specific choice of $\alpha$ to be described later, is the map $Q_0 \to \mathbb{Z}^\ell$ that sends each $x$ in $Q_0$ to the vector $\left(\dim (S^1)_x, \ldots, \dim (S^\ell)_x\right)$.

\subsection*{Zigzag Persistence} Gabriel's foundational theorems from \cite{gabriel1972unzerlegbare} establish that the set of indecomposable representations of a type $\mathbb{A}_n$ quiver can be canonically identified with the collection of subintervals $[u,v] \subset [0,n-1]$ that have integral endpoints. The study of such representations has enjoyed a substantial recent renaissance, induced largely by their appearance in topological data analysis \cite{oudot, zigzag_pers}, where they are called {\em zigzag persistence modules}. Thus, by Gabriel's results, every finite-dimensional zigzag persistence module $P$ decomposes uniquely into a direct sum of the form 
\[
P \simeq \bigoplus_{[u,v]} \I[u,v]^{d_{u,v}};
\] 
here $[u,v] \subset [0,n-1]$ ranges over a finite set $\barc(P)$ called the {\em barcode} of $P$. For each such interval, $\I[u,v]$ is the corresponding indecomposable  whose {\em multiplicity}, denoted $d_{u,v}$ above, is a strictly positive integer. A similar interval decomposition theorem also holds for {\em infinite} zigzag persistence modules \cite{infinite_zigzag}, where the endpoints of intervals are allowed to attain $\pm \infty$ values. Every representation $V$ of a type $\widetilde{\mathbb{A}}_n$ quiver gives rise to an $n$-periodic infinite zigzag module ${\cL}V$, and the second invariant of interest is the associated multiplicity function $\barc({\cL}V) \to \mathbb{N}_{>0}$.

\subsection*{This Paper}

We introduce the {\bf Euler} stability condition $\epsilon:Q_0 \to \mathbb{R}$, which is defined on any finite quiver $Q$ as follows. The underlying graph of $Q$ forms a one-dimensional CW complex $X$; and every representation $V$ of $Q$ functorially induces a cellular sheaf  \cite{curry2014sheaves} ${\bS}V$ on $X$. From this sheaf, one can build a two-term cochain complex which computes the cohomology of $X$ with ${\bS}V$ coefficients. The corresponding Euler characteristic has the form
\[
\chi(X;{\bS}V) = \sum_{x \in Q_0} \left(1-\deg_\text{in}(x)\right) \cdot \dim V_x.
\]
Here $\deg_\text{in}(x)$ equals the number of edges which point to $x$. Our stability condition is therefore given by $\epsilon(x) := 1-\deg_\text{in}(x)$. 

Harder-Narasimhan filtrations along $\epsilon$ enjoy some surprising properties pertaining to barcode decompositions. As a warm-up exercise, we first consider the easiest examples, which are furnished by equioriented quivers of type $\mathbb{A}_n$
\[
\xymatrixcolsep{.6in}
\xymatrix{
x_0 \ar@{->}[r]^-{e_1} & x_1 \ar@{->}[r]^-{e_2} & \cdots \ar@{->}[r]^-{e_{n-1}} & x_{n-1}.
}
\]
The representations of such quivers are called (ordinary, discrete) persistence modules. For any such representation $V$, the Harder-Narasimhan filtration $V^\bullet$ along $\epsilon$ can be used to directly recover the multiplicities of all intervals which have the form $[0,j]$ in $\barc(V)$. Here is a simplified version of Theorem \ref{thm:hn_zerobars}.

\begin{thm*} [\textbf{A}] Let $V$ be a persistence module and let $ j_1 < \ldots < j_\ell$ be the collection of all indices $j$ in $\set{0,\ldots, n-1}$ satisfying $[0,j] \in \barc(V)$. Then $V^\bullet$ has length $\ell+1$, and for each integer $0 < k \leq \ell$ the quotient $S^k := V^k/V^{k-1}$ satisfies
\[
\dim (S^k)_{x_i} = \begin{cases} d_{0,j_k} & i \in [0,j_k], \\
0 & \text{otherwise}.
\end{cases}
\]
\end{thm*}
\noindent Thus, the Harder-Narasimhan filtration of $V$ along $\epsilon$ recovers (the multiplicities of) all those intervals in $\barc(V)$ which have left endpoint $0$. We show that one can also extract similar formulas for multiplicities of all the other intervals in $\barc(V)$ -- and hence, the entire isomorphism class of $V$ -- by availing of Harder-Narasimhan filtrations along $\epsilon$ of $V$'s restriction to certain truncated subquivers. In future work, we will describe a method for recovering multiplicities of all indecomposables (for representations of a different quiver) by using several different stability conditions at once.

Turning now to the central focus of this paper, we consider a representation $V$ of a type $\widetilde{\mathbb{A}}_n$ quiver $Q$ over an algebraically closed field. The indecomposable summands of $V$ have two possible forms, illustrated below:
\begin{center}
    \includegraphics[scale=.35]{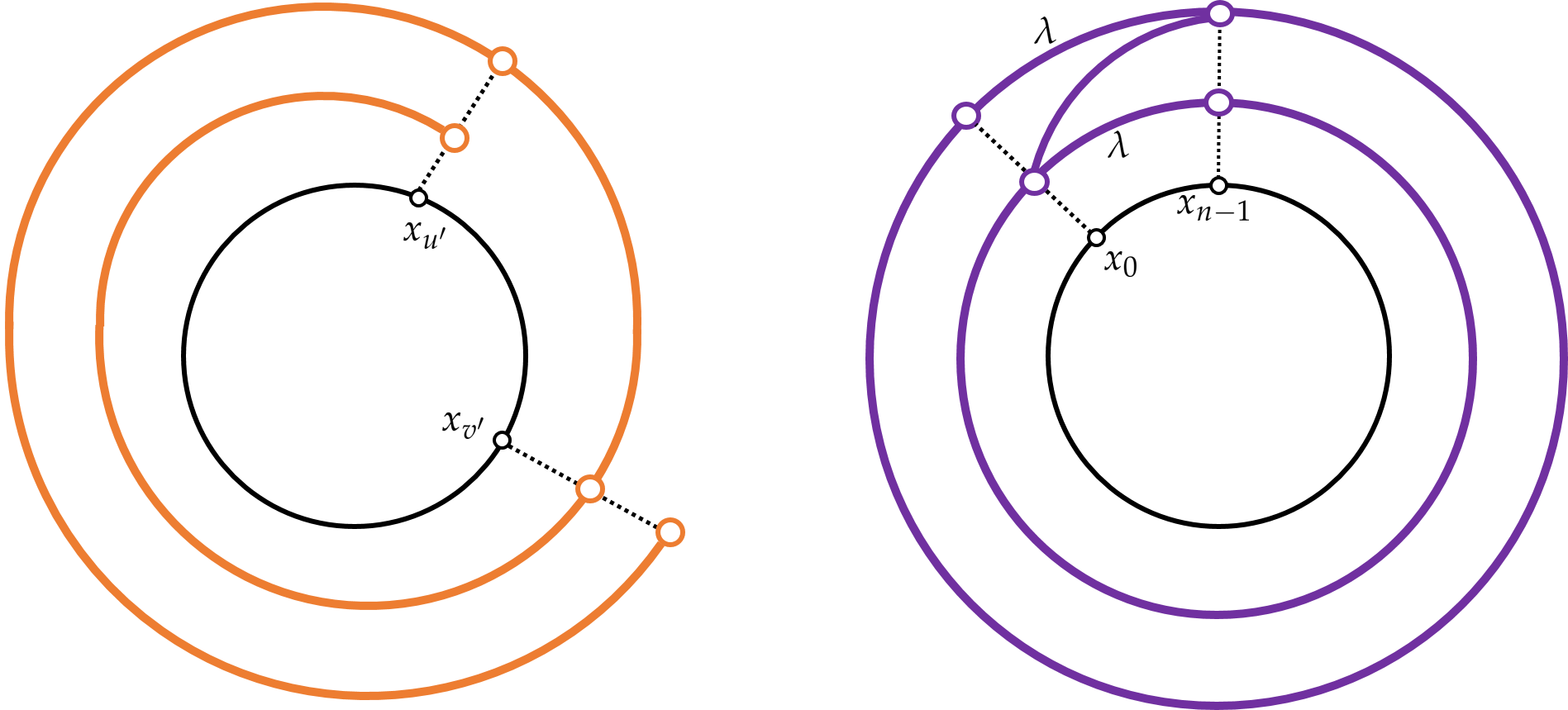}
\end{center}
To the left we have $\N[u,v]$, which is obtained by wrapping an interval $[u,v] \subset \Z$ around $Q$; here $u'$ and $v'$ equal $u$ and $v$ modulo $n$ respectively. To the right lies $\T[\lambda;w]$, where all vector spaces have the same dimension $w \geq 1$ and all edge-maps are identities {\em except} for a Jordan block with diagonal $\lambda \neq 0$ over $e_0$. As mentioned above, $V$ lifts to a representation ${\cL}V$ of the infinite zigzag quiver $\UQ$ obtained by unwinding $Q$. This lift operation $\cL$ forms a functor from the category of representations of $Q$ to the full subcategory of representations of $\UQ$ spanned by $n$-periodic objects. Moreover, ${\cL}V$ decomposes as a direct sum of the form
\begin{align}\label{eq:LVdecomp}
{\cL}V \simeq \I[-\infty,\infty]^{d^*_\infty} \oplus \bigoplus_{[u,v]}\left(\bigoplus_{c \in \Z} \I[u+cn,v+cn]\right)^{d^*_{u,v}}.
\end{align}
The first term on the right consists of infinite intervals corresponding to the $\T[\lambda;w]$ summands of $V$ while the second term consists of (infinitely many) $n$-shifted copies of $\I[u,v]$ corresponding to the $\N[u,v]$ summands. 

Here is a condensed description of our main result, Theorem \ref{thm:main}. 
\begin{thm*} [\textbf{B}]
Let $V$ be a representation of an acyclic type $\widetilde{\mathbb{A}}_n$ quiver $Q$ over an algebraically closed field. Let $V^\bullet$ be the Harder-Narasimhan filtration of $V$ along $\epsilon$, and let $S^j := V^j/V^{j-1}$ be its successive quotients. If $\phi_\epsilon(S^j)$ is nonzero, then for each vertex $x \in Q_0$ we have:
\[
\dim (S^j)_x = \sum_{\N[u,v]} d^*_{u,v} \cdot \dim \N[u,v]_x
\]
where the sum is over all $\N[u,v]$ appearing in the decomposition of $V$ which have the same $\epsilon$-slope as $S^j$. And similarly, if $\phi_\epsilon(S^j)$ equals $0$, then for each vertex $x \in Q_0$, we have
\[
\dim (S^j)_x = d^*_\infty + \sum_{\N[u,v]} d^*_{u,v} \cdot \dim \N[u,v]_x
\]
where the sum is over indecomposable summands $\N[u,v]$ of $V$ that satisfy $\phi_\epsilon(\N[u,v]) = 0$.
\end{thm*}
\noindent Since ${\cL}V$ is periodic, the multiplicities $d^*_\infty$ and $d^*_{u,v}$ of its indecomposable summands can be computed by restricting to a sufficiently long finite zigzag persistence module via the algorithm of \cite{zigzag_pers} --- see Remark \ref{rmk:mult_zigzgag_compute}. We hope that the results of this paper will encourage not only the use of tools from geometric invariant theory in the study of persistence modules, but also facilitate an influx of techniques from topological data analysis for efficient computation of Harder-Narasimhan filtrations.

\subsection*{Organisation} Sections \ref{sec:prelim}, \ref{sec:HN filtr} and \ref{sec:decompo_type_A} contain notation and preliminary material pertaining to quiver representations, their direct sum decompositions into indecomposable representations, and Harder-Narasimhan filtrations respectively. In Section \ref{sec:euler} we introduce the Euler stability condition, and in Section \ref{sec:HN-An} we prove Theorem ({\bf A}). In Section \ref{sec:An-tilde-unwind} we describe the infinite zigzag persistence modules arising as lifts of $\widetilde{\mathbb{A}}_n$ quiver representations. Finally, Section \ref{sec:HN_of_An_tilde} contains the proof of Theorem ({\bf B}).

\subsection*{Related Work}

The recent work of Kinser \cite{kinser}  classifies {\em totally stable} conditions on type $\mathbb{A}$ quivers, for which every indecomposable is stable, rather than semistable\footnote{In this context, a $Q$-representation $V$ is called {\em stable} with respect to $\alpha:Q_0 \to \R$ if the strict inequality $\phi_\alpha(U) < \phi_\alpha(V)$ holds for every subrepresentation $0 \subsetneq U \subsetneq V$.}. Using such a stability condition instead of $\epsilon$ in Section \ref{sec:HN-An} would allow us to recover the entire barcode at once from the HN filtration (rather than only the intervals with left endpoint $0$). For more complicated quivers, however, the set of totally stable conditions is always empty. In \cite{stable_quiver}, Hille and de la Pe\~na  characterise stable representations of tame quivers when the stability condition is in a neighbourhood of a quantity called the {\em defect}. For $\widetilde{\mathbb A}_n$ quivers, the defect happens to coincide with the Euler slope, and hence their work yields a different proof of our Proposition \ref{prop:ss_An_tilde}. Unlike their argument, ours does not use any knowledge of tame hereditary algebras (besides the list of indecomposable type  $\widetilde{\mathbb A}$ quiver representations). More recently, Apruzzese and Igusa  \cite{igusa} have used a geometric model to determine the maximum finite number of stable indecomposables of type $\widetilde{\mathbb A}$ quivers as the stability condition is varied.  

 The idea of relating indecomposables of a quiver to those of its universal cover, which is exploited heavily in Section \ref{sec:An-tilde-unwind}, dates back to the work of Riedtmann, \cite{cov-ried} Bongartz-Gabriel \cite{cov-bongab}, and Gabriel \cite{cov-gabriel}. An algorithmic perspective on covering theory for representations of strictly alternating $\widetilde{\mathbb{A}}_n$ quivers has been employed in the work of Burghelea and Dey on circle-valued persistence \cite{burghelea}, where indecomposables of the form $\T[\lambda;w]$ are called {\em Jordan cells}. Similarly, Cheng \cite{compute_HN} presents a polynomial time algorithm  to compute Harder-Narasimhan filtrations with respect to any stability condition by relating them to the so-called {\em discrepancies} of quiver representations. Computing discrepancies requires finding the largest $c$ such that a given space of matrices has a {\em $c$-shrunk subspace}. Although this last problem admits  a  polynomial-time algorithm \cite{ivanyos2018constructive}, we are not aware of any practical implementations.

Finally, the work of Henselman and Ghrist \cite{saecular} seeks to generalise persistence barcodes for functors to non-abelian  categories by viewing (ordinary) persistence modules as lattice homomorphisms from a lattice of intervals with respect to a certain partial order. In contrast, the $\epsilon$-slope defines a total preorder on the set of intervals. If the interval modules are all stable (i.e., if the stability condition is totally stable in the sense of  \cite{kinser}), then this preoreder is an order and the HN filtration constitutes a {\em subsaecular series} in the language of \cite{saecular}.

{
\subsection*{Acknowledgements} We thank Fabian Haiden, Emile Jacqard and Frances Kirwan for several helpful conversations. The authors are members of the Centre for Topological Data Analysis, funded by the EPSRC grant EP/R018472/1. VN is partially supported by US AFOSR grant FA9550-22-1-0462.}

\section{Quiver Representation Preliminaries}\label{sec:prelim}

The study of quiver representations is a vast enterprise spanning algebra and geometry, so we will restrict our focus here on the aspects relevant to this work; comprehensive treatments can be found in \cite{kirillov_quiver} and \cite{schiffler2014quiver}.

A \emph{quiver} $Q$ consists of two sets $Q_0$ and $Q_1$, whose elements are called vertices and edges respectively, equipped with two functions $s,t:Q_1 \to Q_0$ called the source and target map. We typically write $e: x \to y$ when indicating that the edge $e$ has {source} $s(e) = x$ and {target} $t(e) = y$. A {\em path} of $Q$ is any nonempty finite sequence of edges $p = (e_1,\ldots,e_k)$ so that $s(e_j) = t(e_{j-1})$ holds across all $1 < j \leq k$. The source and target maps extend to any such $p$ by setting $s(p) = s(e_1)$ and $t(p) = t(e_k)$, and we call $p$ a {\em cycle} of $Q$ whenever the source and target vertex of $p$ are identical. The quiver $Q$ is called {\bf acyclic} if it does not contain any such cycles.

Let us fix, once and for all, a ground field $\K$ so that all vector spaces and linear maps encountered henceforth are understood to be defined over $\K$. A \emph{ representation} $V$ of $Q$ is an assignment of 
  \begin{enumerate}
      \item a vector space $V_x$ to each vertex $x\in Q_0$, and 
      \item a linear map $V_{e}:V_x\to V_y$ to each edge $e:x \to y$ in $Q_1$.
  \end{enumerate}
 Unless stated otherwise, we will assume that both $Q_0$ and  $Q_1$ are finite, and we will only consider finite-dimensional representations of $Q$ --- i.e., each $V_x$ is required to be finite-dimensional over $\K$. The map $\udim_V:Q_0 \to \Z$ sending each $x$ to $\dim V_x$ is called the {\em dimension vector} of $V$. 
  
  A \emph{ morphism} of representations $f:V \to V'$  is a collection of vertex-indexed linear maps $\set{f_x:V_x \to V'_x \mid x \in Q_0}$ so that for each edge $e:x \to y$ in $Q_1$ the evident diagram of vector spaces commutes:
  \[
  \xymatrixcolsep{1in}
  \xymatrixrowsep{.4in}
  \xymatrix{
  V_x \ar@{->}[r]^{f_x} \ar@{->}[d]_{V_e} & V'_x \ar@{->}[d]^{V'_e} \\
  V_y \ar@{->}[r]_{f_y} & V'_y
  }
  \]
  Equivalently, the identity $f_{t(e)} \circ   V_{e}=V_{e}'\circ f_{s(e)}$ holds for every edge $e\in Q_1$. With this definition of morphisms in place, the representations of $Q$ define an abelian category $\Rep(Q)$, see \cite[Section 1.2]{gabriel1972unzerlegbare}. Injective and surjective morphisms, kernels, images, quotients, direct sums and the zero object are all defined pointwise in $\Rep(Q)$. We say that $V$ is a {\em subrepresentation} of another representation $V'$ whenever there exists an injective morphism $V \hookrightarrow V'$, in which case we simply write $V \subset V'$. 
  
\begin{rmk}  The quiver $Q$ is a {\em subquiver} of another quiver $Q'=(s',t':Q_1' \to Q_0')$ if we have $Q_0\subset Q_0'$ and $Q_1\subset Q_1'$ so that $s$ and $t$ are restrictions of $s'$ and $t'$ respectively. Given such a subquiver, we note that each representation $V'$ of $Q'$ automatically induces a representation $V$ of $Q$ by restricting to the available vertices and edges. Namely, $V_x := V'_x$ for all $x\in Q_0$ and $V_{e} := V'_{e}$ for all $e\in Q_1$.
\end{rmk}
  
  \section{Harder-Narasimhan Filtrations} \label{sec:HN filtr}
  
Harder-Narasimhan filtrations were originally introduced for the purpose of studying moduli spaces of vector bundles over algebraic curves \cite{Harder1974}; they have since been employed in a plethora of other contexts \cite{king1994moduli, bridgeland2007stability, haiden2020semistability}, including geometric invariant theory \cite[Chapter 4.2]{gitbook}. 
 
The {\em Grothendieck group} $K(\cA)$ of an abelian category $\cA$ is the abelian group generated by the objects of $\cA$ subject to the relation that $V = U + W$ whenever there is an exact sequence $0 \to U \to V \to W \to 0$. By a {\bf stability condition} on $\cA$ we mean an abelian group homomorphism $Z:K(\cA) \to (\C,+)$ valued in the (additive) complex numbers, so that every nonzero object $U$ is mapped to the right half plane, \ie $\Real Z(U) > 0$. Given such a stability condition, the \emph{$Z$- slope} of a nonzero object $V$ in $\cA$ is the real number
\[
\phi_Z(V):=\frac{\Imag  Z(V)}{\Real Z(V)}.
\]

\begin{defi} Let $Z$ be a stability condition on $\mathcal{A}$. A nonzero object $V$ in $\mathcal A$  is {\bf $Z$- semistable} if its nonzero subobjects have smaller slopes. In other words, $\phi_Z(U)\leqslant \phi_Z(V)$ holds whenever there exists an injective morphism $U \hookrightarrow V$ in $\mathcal{A}$ with $U \neq 0$. 
\end{defi}

The importance of semistable objects in the study of moduli spaces stems from the fact that every nonzero object admits a unique filtration for which the successive quotients are semistable and have strictly decreasing slopes. A proof of the following result (for the category of representations of a fixed quiver) can be found in \cite[Theorem 2.5]{stable_quiver}; more generally, see \cite[Theorem 4.2]{haiden2020semistability}.

\begin{thm}\label{thm:HN_filtration}
Let $\cA$ be a finite length (i.e., Noetherian and Artinian) abelian category equipped with a stability condition $Z:K(\cA) \to (\C,+)$. Every object $V \neq 0$ of $\cA$ admits a unique filtration $V^\bullet := {\bf HN}_\alpha^\bullet(V)$: 
\[
0 = V^0 \subsetneq V^1 \subsetneq \cdots \subsetneq V^\ell = V
\]
whose successive quotients $V^j/V^{j-1}$ are all $\alpha$-semistable with strictly decreasing $\alpha$-slopes, i.e., 
\[
\phi_Z(V^{j}/V^{j-1}) > \phi_Z(V^{j+1}/V^{j}). 
\] 
(This unique $V^\bullet$ is called the {\bf Harder-Narasimhan}, or {\bf HN}, filtration of $V$ along $Z$.)
\end{thm}

\begin{rmk}\label{rem:hnss}
 In particular, it follows from uniqueness that a nonzero object $V$ of $\cA$ is $Z$-semistable if and only if the corresponding HN filtration is the trivial one $0 \subsetneq V$.
\end{rmk}

Let $\Hom(Q_0,\Z)$ be the abelian group consisting of functions $Q_0 \to \Z$ with addition defined vertex-wise.
When $Q$ is acyclic, the map $V \mapsto \udim_V$ furnishes an isomorphism $K(\Rep(Q)) \simeq \Hom(Q_0,\Z)$  --- see for instance \cite[Theorem 1.15]{kirillov_quiver}. As a result, every stability condition on $\Rep(Q)$ amounts to a function $\beta:Q_0 \to \mathbb{C}$ sending vertices $x$ to complex numbers $\beta_x$, subject to the requirement that $\Real \beta_x > 0$. It is customary to assign numbers of the form $\beta_x = 1+i \cdot \alpha_x$ for arbitrary real numbers $\alpha_x$ \cite{stable_quiver,Reineke_2003}, which makes no difference to semistability (but is liable to alter HN filtrations). Thus, the stability conditions of interest are precisely the functions $\alpha:Q_0 \to \mathbb{R}$, and for such a function the corresponding $\alpha$-slope of $V \in \Rep(Q)$ equals
\[
\phi_\alpha(V)= \frac{\sum_{x \in Q_0} \alpha_x \cdot \dim V_x}{\sum_{x \in Q_0} \dim V_x}.
\]

Crucially, the uniqueness result of Theorem \ref{thm:HN_filtration} applies only after the stability condition has been fixed --- in particular, varying $\alpha:Q_0 \to \mathbb{R}$ is liable to produce a very different HN filtration of the same representation. If $\alpha$ is identically zero for instance, then all quiver representations are semistable with slope $0$, and hence have trivial HN filtrations. Our focus here will be on a specific choice of stability condition $\epsilon$ arising from cellular sheaf cohomology, which is described in Section \ref{sec:euler}. Throughout the remainder of this section, however, we fix an arbitrary $\alpha$ and all instances of slopes, stability and HN filtrations encountered here are with respect to this fixed $\alpha$. 

The following result was established for vector bundles in \cite[\S  2]{HN_vect_bundle}; the version stated below is \cite[Lemma 2.2]{Reineke_2003}.

\begin{lmm}\label{lmm:seq_and_slope}
The following properties hold for every exact sequence 
\[
0\rightarrow U\rightarrow V\rightarrow W\rightarrow 0
\] in $\Rep(Q)$ with $U,V$ and $W$ all nonzero:
\begin{enumerate}
    \item $\min(\phi(U),\phi(V))\leqslant\phi(X)\leqslant \max(\phi(U),\phi(V))$; and moreover,
    \item if $\phi(U)=\phi(V)$ with $U$ and $V$ semistable, then $X$ is also semistable.
\end{enumerate}
\end{lmm}

\noindent It will be convenient for our purposes to describe the HN filtrations of direct sums in terms of the HN filtrations of constituent factors. To this end, given any representation $V$ in $\Rep(Q)$, we consider the HN filtration $V^\bullet$ as being indexed over the real line in the following manner. The {\em HN $\R$-filtration} of $V$ is the assignment $t \mapsto V_\R(t)$ of quiver representations to real numbers obtained by setting
\[
V_\R(t) := V^{\lfloor t \rfloor_\phi};
\]
here $\lfloor t \rfloor_\phi$ is the smallest integer $i \geq 0$ satisfying $t \geq \phi(V^{i+1}/V^{i})$, as illustrated below:
 \[
\footnotesize\begin{tikzpicture}[
roundfillnode/.style={circle,fill, inner sep=1.25pt},scale=1,thick,->]
    \node[label=$V^2$] (0) at (0.5,0) {};
         \node[roundfillnode,label=below:$\phi(V^2/V^1)$]   at (2,0) {};
    \node[label=$V^1$]  (4) at (4,0) {};
             \node[roundfillnode,label=below:$\phi(V^1/V^0)$]   at (6,0) {};
     \node[label=$ V_\R(t)$]  (end) at (10,0) {};
     \node[label=below:$t$]  (end2) at (10,0.) {};
         \node[label=$V^0$] (0) at (8,0) {};
      \draw (-1,0.) to  (10,0.);  
            \node  at (-1.5,0.) {$\dots$}; 
        \end{tikzpicture}
\]
Given $s \geq t$ in $\R$, we have $\lfloor s \rfloor_\phi \leq \lfloor t \rfloor_\phi$, so there is an obvious inclusion $V_\R(t) \hookrightarrow V_\R(s)$. Thus, $V_\R$ constitutes a functor from the $\geq$-ordered set of real numbers to the category $\Rep(Q)$. There is a natural direct sum operation on such functors, where $[U_\R \oplus V_\R](t)$ equals $U_\R(t) \oplus V_\R(t)$ for each $t \in \R$. 

The following result is  well-known to experts, but we were unable to find it in the literature and have included a proof for completeness.

\begin{prop} \label{prop:direct_sum_HN} For nonzero $U$ and $V$ in $\Rep(Q)$, we have $(U \oplus V)_\R = U_\R \oplus V_\R$.
\end{prop}

\begin{proof}
Writing $\ell$ and $m$ for the lengths of $U^\bullet$ and $V^\bullet$ respectively, let 
\[
\Phi := \left\{\phi(U^i/U^{i-1})\right\}_{1 \leq i \leq \ell} \cup \left\{\phi(V^{j}/V^{j-1})\right\}_{1 \leq j \leq m}
\]
be the union of slopes attained by the successive quotients of both filtrations. By construction, $(U \oplus V)_\R$ is locally constant on $\R - \Phi$. Let $\theta_1 > \ldots > \theta_n$ be the slopes in $\Phi$ arranged in strictly decreasing order, and consider the filtration $W^\bullet$ of $U \oplus V$ given by $W^k := [U_\R \oplus V_\R](\theta_k)$ for each $1 \leq k \leq m$. We now claim that $W^\bullet$ is the HN filtration of $(U \oplus V)$. To see this, note that we have an isomorphism of quotient representations
\begin{align}\label{eq:hnsum}
\frac{W^k} {W^{k-1}} \simeq 
\frac{U_\R(\theta_k)}{U_\R(\theta_{k-1})} \oplus \frac{V_\R(\theta_k)}{V_\R(\theta_{k-1})}
\end{align}
for each $k$. We will now show that $W^k/W^{k-1}$ is semistable with slope $\theta_k$, which -- when combined with the uniqueness guarantee of Theorem \ref{thm:HN_filtration} -- produces the desired result. By \eqref{eq:hnsum}, we have a (split) exact sequence of the form  
\[
0 \to \frac{U_\R(\theta_k)}{U_\R(\theta_{k-1})} \hookrightarrow \frac{W^k}{W^{k-1}} \twoheadrightarrow \frac{V_\R(\theta_k)}{V_\R(\theta_{k-1})} \to 0
\]
The two summands on the right side of \eqref{eq:hnsum} are either trivial or semistable with slope $\theta_k$ by construction. In the nontrivial case, the first assertion of Lemma \ref{lmm:seq_and_slope} ensures that the slope of $W^k/W^{k-1}$ is also $\theta_k$, while the second assertion guarantees semistability. 
\end{proof}

\section{Indecomposables and Barcodes}\label{sec:decompo_type_A}

A representation $V \neq 0$ of a quiver $Q$ is \emph{indecomposable} if it cannot be written as a direct sum of two nontrivial representations. Since we have assumed that $Q$ is finite and $V$ is finite-dimensional, the  Krull-Schmidt theorem \cite{krull2, krull} applies, so there is a unique pair consisting of a finite set $\text{Ind}(V)$ of indecomposable representations and a function $\text{Ind}(V) \to \mathbb{N}_{>0}$ denoted $I \mapsto d_I$ that satisfies
\[
V = \hspace{-.2em}\bigoplus_{I \in \text{Ind}(V)} \hspace{-.2em} I^{d_I}.
\]
It follows from Proposition \ref{prop:direct_sum_HN} that the HN filtrations of all $I$ in $\text{Ind}(V)$ determine the HN filtration of $V$, so in principle it suffices to restrict attention to indecomposable representations. Unfortunately, the task of decomposing a given $V \in \Rep(Q)$ into indecomposables, and the task of cataloguing all possible indecomposables in $\Rep(Q)$ are both remarkably difficult problems for arbitrary $Q$. Two prominent exceptions are the Dynkin quivers of type $\mathbb{A}$ and $\widetilde{\mathbb{A}}$, whose indecomposables we describe below.

\subsection{Quivers of type $\mathbb{A}$} \label{sec:An_decomp} Gabriel's seminal result \cite{gabriel1972unzerlegbare} established that $\Rep(Q)$ is {\em representation finite}\footnote{Representation finiteness of $Q$ means that there are only finitely many indecomposables in $\Rep(Q)$ which attain a given dimension vector.} if and only if the undirected graph obtained from $Q$ by ignoring source and target information is a disjoint union of simply-laced Dynkin diagrams  (i.e., of types $\mathbb{A}, \mathbb{D}$ and $\mathbb{E}$). In particular, we recall that a quiver $Q$ is said to be of {\em type $\mathbb{A}_n$} for some $n \geq 0$ whenever its underlying graph is:
\[
\xymatrixcolsep{.6in}
\xymatrix{
x_0 \ar@{-}[r]^-{e_1} & x_1 \ar@{-}[r]^-{e_2} & \cdots \ar@{-}[r]^-{e_{n-1}} & x_{n-1}. 
}
\]
Representations of type $\mathbb{A}_n$ quivers are also called {\bf zigzag persistence modules} \cite{zigzag_pers}. If $s(e_i) = x_{i-1}$ and $t(e_i) = x_i$ holds for each $i$, then $Q$ is called {\em equioriented}, and its representations are (ordinary) {\bf persistence modules}  \cite{zomorodian2004computing}. 

\begin{defi} \label{def:interval_modules} Let $Q$ be a quiver of type $\mathbb{A}_n$ and consider a subinterval $[u,v] \subset [0,n-1]$ with $u$ and $v$ integers. The {\bf interval module} $\I[u,v]$ is the representation of $Q$ defined as follows. The vector spaces assigned to vertices are 
\[\I[u,v]_{x_i}=\begin{cases}\K &\text{ if } i \in [u,v]\\ 0&\text{ otherwise.}\end{cases}
\]
The linear map $\I[u,v]_{e_j}$ assigned to the edge $e_j$ is the identity $\id_\K$ whenever both source and target vector spaces are nontrivial, and is necessarily  zero otherwise.
\end{defi}

The following result is a direct consequence of Gabriel's theorem \cite{gabriel1972unzerlegbare}.

\begin{thm}\label{thm:gabriel_An} Let $Q$ be a type $\mathbb{A}_n$ quiver. Then up to isomorphism, the indecomposables of $\Rep(Q)$ are precisely the interval modules $\I[u,v]$ for $[u,v] \subset [0,n-1]$.
\end{thm}
 It follows directly from the above result that for every zigzag persistence module $V$ there exists a unique finite set $\barc(V)$ containing subintervals of $[0,n-1]$ and a unique function $\barc(V) \to \mathbb{N}_{>0}$ denoted $[u,v] \mapsto d_{uv}$ which satisfy the following property: there is an isomorphism
\begin{align}\label{eqn:decompo_An}
V \simeq \hspace{-.3em}\bigoplus_{[u,v] \in \barc(V)}\hspace{-.35em} \I[u,v]^{d_{uv}}
\end{align}
in $\Rep(Q)$. This set $\barc(V)$ is called the {\bf barcode} of $V$ while $[u,v] \mapsto d_{uv}$ is called the {\bf multiplicity} function. There are practical algorithms which can compute both barcodes and multiplicities for zigzag persistence modules --- see \cite{zigzag_pers}.

In this paper we will also be interested in representations of type $\mathbb{A}_\infty$ quivers; the underlying graph of any such quiver $Q$ has vertices indexed by $\Z$ as depicted below:
\[
\xymatrixcolsep{.6in}
\xymatrix{
\cdots \ar@{-}[r]^{e_{-1}} & x_{-1} \ar@{-}[r]^{e_{0}} & x_0 \ar@{-}[r]^{e_{1}} & x_1 \ar@{-}[r]^{e_{2}} & \cdots
}
\]
Interval modules $\I[u,v]$ in $\Rep(Q)$ are described as in Definition \ref{def:interval_modules}, except that in addition to the usual subintervals $[u,v] \subset \Z$, one also allows infinite intervals with $u = -\infty$ or $v=\infty$ or both. The following result summarises \cite[Theorem 1.6 and Theorem 1.7]{infinite_zigzag}.

\begin{thm}\label{thm:A-inf-indec}
Let $Q$ be a type $\mathbb{A}_\infty$ quiver. Then, 
\begin{enumerate}
    \item the only indecomposable objects in $\Rep(Q)$ are the interval modules; and moreover, 
    \item every $V \in \Rep(Q)$ satisfying $\dim V_x < \infty$ for all $x \in Q_0$ admits a unique (and not necessarily finite) set $\barc(V)$ consisting of subintervals $[u,v] \subset \Z \cup \set{\pm \infty}$ and a multiplicity function $d:\barc(V) \to \Z_{>0}$ so that the decomposition
\[
V \simeq \bigoplus_{[u,v] \in \barc(V)} \I[u,v]^{d_{uv}}
\]
holds in $\Rep(Q)$. 
\end{enumerate}
\end{thm}
    
\subsection{Affine Quivers of type $\widetilde{\mathbb{A}}$}\label{subsec:An_tilde} We say that $Q$ is of {\em type $\widetilde{\mathbb{A}}_n$} if its underlying graph has the form
\begin{align}\label{eq:Antilde}
\xymatrixcolsep{.6in}
\xymatrixrowsep{.45in}
\xymatrix{
& & x_{n-1} \ar@{-}[dll]_-{e_0} \ar@{-}[drr]^-{e_{n-1}} & & \\
x_0 \ar@{-}[r]_-{e_1} & x_1 \ar@{-}[r]_-{e_2} & \cdots  \ar@{-}[r]_-{e_{n-3}} & x_{n-3} \ar@{-}[r]_-{e_{n-2}} & x_{n-2}. 
}
\end{align}
for $n > 1$. In this section we assume that the ground field $\mathbb{K}$ is algebraically closed. With this assumption in place, the indecomposables of $Q$ can be classified into two families, which we describe below (with illustrative examples to follow).

\begin{defi}\label{def:indAtilde}
Let $Q$ be a quiver of type $\widetilde{\mathbb{A}}_n$.
\begin{enumerate}
    \item For each interval $[u,v] \subset \mathbb{Z}$, the representation $\N[u,v]$ of $Q$ is defined as follows. Setting $\ell := \lfloor \frac{v-u}{n}\rfloor$, the vector space assigned to $x_j$ is 
    \[
    \N[u,v]_{x_j} = \begin{cases}
    \mathbb{K}^{\ell+1} & \text{ if } j \in [u,v] \text{ mod } n, \text{ and } \\
    \mathbb{K}^\ell & \text{ otherwise.} 
    \end{cases}
    \]
    The linear map over $e_j$ is the identity whenever its source and target vector spaces are equidimensional; the exceptional cases occur when $j$ equals either $u$ or $v+1$ modulo $n$. In such cases, if $e_j$ has a clockwise orientation, then we have:
    \[
    \N[u,v]_{e_j} = \begin{cases} 
                        [0~\text{id}_\ell]  & \text{if } (u \text{ mod } n) = j \neq (v+1 \text{ mod } n), \\
                        [\text{id}_\ell~0]^T  & \text{if } (u \text{ mod } n) \neq j = (v+1 \text{ mod } n), \text{ and }\\
                       \left[\begin{smallmatrix}0&\text{id}_\ell\\0 &0\end{smallmatrix}\right]&\text{if } (u \text{ mod } n) = j = (v+1 \text{ mod } n)
                    \end{cases}.
    \] 
Similarly, when $e_j$ has a counterclockwise orientation,  $   \N[u,v]_{e_j} $ is the transpose of the appropriate matrix above.
    \item For each field element $\lambda \neq 0$ in $\mathbb{K}$ and integer $w \geq 1$, let $\T[\lambda;w]$ be the $Q$-representation which assigns to every vertex the vector space $\mathbb{K}^w$, to every edge $e_j$ for $j \neq 0$ the identity map $\text{id}_w$, and to $e_{0}$ the Jordan block with $\lambda$ along its diagonal.
\end{enumerate}
\end{defi}
We observe that $\N[u,v]$ and $\N[u',v']$ are isomorphic whenever both (a) lengths $v-u$ and $v'-u'$ coincide, and (b) the equalities $u = u'$ and $v = v'$ hold modulo $n$. To remove the ambiguity, we will implicitly assume henceforth that $u$ lies in $\set{0,\ldots,n-1}$ for every $\N[u,v]$.

\begin{ex} \label{ex:An-tilde-reps} Here is the representation $\N[1,9]$ of a type $\widetilde{\mathbb{A}}_6$ quiver; note that the dimension is larger between vertices $x_1$ and $x_3$, and that unlabelled arrows carry identity maps:
\[
\xymatrixcolsep{.6in}
\xymatrixrowsep{.45in}
\xymatrix{
& & \mathbb{K} \ar@{->}[dll]& & \\
\mathbb{K} \ar@{->}[r]_{\left[\begin{smallmatrix}0\\1\end{smallmatrix}\right]}& \mathbb{K}^2  \ar@{->}[r] & \mathbb{K}^2   & \mathbb{K}^2  \ar@{->}[r]_-{[1 ~ 0]} \ar@{->}[l] & \mathbb{K} \ar@{->}[ull]
}
\]
Similarly, here is the representation $\T[2;3]$ of the same quiver:
\[
\xymatrixcolsep{.6in}
\xymatrixrowsep{.45in}
\xymatrix{
& & \mathbb{K}^3 \ar@{->}[dll]_{\left[\begin{smallmatrix}2 & 1 & 0\\0 & 2 & 1 \\ 0 & 0 & 2\end{smallmatrix}\right]}& & \\
\mathbb{K}^3 \ar@{->}[r] & \mathbb{K}^3  \ar@{->}[r] & \mathbb{K}^3   & \mathbb{K}^3 \ar@{->}[l] \ar@{->}[r]  & \mathbb{K}^3 \ar@{->}[ull]
}
\]
\end{ex}

A proof of the following result can be found in \cite{donovan1973representation}.
 \begin{thm}\label{thm:an-tilde-indec}
 A representation of a type $\widetilde{\mathbb{A}}_n$ quiver is indecomposable if and only if it is isomorphic either to $\N[u,v]$ for some interval $[u,v]  \subset \Z$ or to $\T[\lambda;w]$ for some $\lambda \neq 0$ in $\K$ and $w \geq 1$ in $\mathbb{Z}$.
 \end{thm}

\section{The Euler Stability Condition}\label{sec:euler}

The stability condition that we will use throughout this paper is simple to define. By the discussion following Remark \ref{rem:hnss}, we only require a function $Q_0 \to \mathbb{R}$. This function is 
\[
x \mapsto \left(1-\#\set{e \in Q_1 \mid t(e) = x}\right),
\] 
with $\#$ denoting cardinality. Readers who are satisfied with this definition may safely proceed to the next section; in this section we will only describe the reasons which have motivated our choice. Consider a CW complex $X$, write $\sigma \leq \tau$ to indicate that the cell $\sigma$ lies in the boundary of the cell $\tau$ in $X$, and denote the poset of cells ordered by this face relation by $(X,\leq)$. A {\bf cellular sheaf} $F$ on $X$ is a functor from $(X,\leq)$ to the category $\Vect(\K)$ of $\mathbb{K}$-vector spaces \cite{curry2014sheaves}. Thus, $F$ assigns to each cell $\sigma$ a vector space $F_\sigma$ and to each face relation $\sigma \leq \tau$ a linear map $F_{\sigma\tau}:F_\sigma \to F_\tau$, subject to two axioms:
\begin{enumerate}
    \item (identity) the map $F_{\sigma\sigma}:F_\sigma \to F_\sigma$ is the identity for each cell $\sigma$, and
    \item (associativity) across any triple $\sigma \leq \tau \leq \nu$ of cells, we have $F_{\sigma\nu} = F_{\tau\nu} \circ F_{\sigma\tau}$.
\end{enumerate}
When $X$ is one-dimensional, the associativity axiom holds automatically since there are no triples of the form $\sigma < \tau < \nu$. Cellular sheaves on $X$ form an abelian category $\textbf{Shf}(X)$ whose morphisms are given by the natural transformations between functors $(X,\leq) \to \Vect(\K)$.

Let $F$ be a cellular sheaf on a CW complex $X$, and consider the sequence of vector spaces and linear maps
\begin{align}\label{eq:sheafcoh}
\xymatrixcolsep{.45in}
\xymatrix{
0 \ar@{->}[r] & C^0(X;F) \ar@{->}[r]^-{\delta_F^0} & C^1(X;F) \ar@{->}[r]^-{\delta_F^1} & \cdots,
}
\end{align}
where $C^j(X;F) := \prod_{\dim \sigma = j} F_\sigma$ for each dimension $j$, and the map $d_F^j$ has the following block structure. Its component $F_\sigma \to F_\tau$ equals 
\[
\delta_F^j|_{\sigma\tau} := \begin{cases} [\sigma:\tau] \cdot F_{\sigma\tau} & \text{if } \sigma \leq \tau\\
0 & \text{otherwise},
\end{cases}
\]
with $[\sigma:\tau]$ denoting the degree of the attaching map of $\tau$'s boundary along $\sigma$. A simple calculation \cite[Lemma 6.2.2]{curry2014sheaves} confirms that \eqref{eq:sheafcoh} is a cochain complex, and the associated cohomology is denoted
$H^j(X;F) := {\ker \delta^j_F}\big/{\text{img }\delta^{j-1}_F}$. If these {\em sheaf cohomology} groups are finite dimensional for all $j$ and vanish for $j \gg 0$, then there is a well-defined {\bf Euler characteristic} of $F$:
\[
\chi(X;F) := \sum_{j \geq 0} (-1)^j \cdot \dim H^j(X;F),
\]
which (as usual) also equals the alternating sum $\sum_{j \geq 0} (-1)^j \cdot \dim C^j(X;F)$ whenever this sum makes sense.

The graph underlying any quiver may be treated as a one-dimensional regular CW complex --- cells of dimension 0 and 1 are $Q_0$ and $Q_1$ respectively. In this case, the zeroth sheaf cohomology $H^0(X;{\bS}V)$ coincides with the {\em space of sections} \cite[Section 1]{seigal2022} of $V$.

\begin{defi}\label{def:repsheaf}
Let $Q$ be a quiver with underlying graph $X$. The {\bf associated sheaf functor} $\bS:\Rep(Q) \to \text{\bf Shf}(X)$ sends each $V$ in $\Rep(Q)$ to the cellular sheaf ${\bS}V$ in $\text{\bf Shf}(X)$ given by:
\[
({\bS}V)_\sigma:=\left\{\begin{array}{ll}
V_\sigma &\text{if }\sigma\in Q_0 \\
 V_{t(\sigma)}&\text{if }\sigma\in Q_1 \\
\end{array}\right. \quad\text{and }\quad ({\bS}V)_{\sigma\tau} :=\left\{\begin{array}{ll}
        V_\tau & \text{if }\sigma=s(\tau) \\
        \text{id} & \text{if }\sigma=t(\tau) 
    \end{array}\right.
\]
Morphisms $V \to W$ in $\Rep(Q)$ induce morphisms ${\bS}V \to {\bS}W$ in $\text{\bf Shf}(X)$ in a straightforward manner.
\end{defi}

Our next result shows that the Euler characteristic $\chi(X;{\bS}V)$ can form the imaginary part of a stability condition on $\Rep(Q)$.

\begin{prop}
Let $Q$ be a quiver with underlying graph $X$. The map $V \mapsto \chi(X;{\bS}V)$ is a well-defined group homomorphism $K_0(\Rep(Q)) \to (\mathbb{R},+)$.
\end{prop}
\begin{proof}
It is not difficult to confirm that the functor $\bS$ is exact, so every exact sequence $0 \to U \to V \to W \to 0$ in $\Rep(Q)$ produces a short exact sequence of cochain complexes
\[
0 \to C^\bullet(X;{\bS}U) \to C^\bullet(X;{\bS}V) \to C^\bullet(X;{\bS}W) \to 0.
\]
Passing to the long exact sequence on sheaf cohomology, we obtain the desired equality $\chi(X;{\bS}V) = \chi(X;{\bS}U) + \chi(X;{\bS}W)$.
\end{proof}

Since there are only two nontrivial cochain groups to consider in \eqref{eq:sheafcoh} when $X$ is the underlying graph of a quiver $Q$, we have for each $V$ in $\Rep(Q)$ the formula
\begin{align*}
\chi(X;{\bS}V) &= \sum_{x \in Q_0}\dim V_{x} - \sum_{e \in Q_1} \dim V_{t(e)} & \text{by }\eqref{eq:sheafcoh} \text{ and Def }\ref{def:repsheaf}, \\
&= \sum_{x \in Q_0} \left(1-\deg_\text{in}(x)\right) \cdot \dim V_x & .
\end{align*}
Here $\deg_\text{in}(x) := \#\set{e \in Q_1 \mid t(e)=x}$. We have arrived at the desired stability condition.

\begin{defi}\label{def:eulerstab}
The {\bf Euler stability condition} for an acyclic quiver $Q$ is the map $\epsilon:Q_0 \to \mathbb{R}$ given by $
\epsilon(x) := (1-\deg_\text{in}(x))$, and the associated slope defined for $V \neq 0$ in $\Rep(Q)$ by
\[
\phi_\epsilon(V) := \frac{\sum_{x \in Q_0} \epsilon(x) \cdot \dim V_x}{\sum_{x \in Q_0} \dim V_x},
\]
is called the {\bf Euler slope}.
\end{defi}

\section{HN Filtrations of Ordinary Persistence Modules}\label{sec:HN-An}

Throughout this section, $Q$ will denote the equioriented quiver of type $\mathbb{A}_n$ for some $n\geq 1$, and $V$ will denote a fixed nonzero object of $\Rep(Q)$, i.e., an ordinary persistence module: 
\begin{align}\label{eq:pmod}
\xymatrixcolsep{.6in}
\xymatrix{
V_0 \ar@{->}[r]^-{f_1} & V_1 \ar@{->}[r]^-{f_2} & \cdots \ar@{->}[r]^-{f_{n-1}} & V_{n-1}}.
\end{align} 
Our goal here is to link the barcode decomposition of $V$ from (\ref{eqn:decompo_An}), recalled below:
\[
V \simeq \hspace{-.3em}\bigoplus_{[u,v] \in \barc(V)}\hspace{-.35em} \I[u,v]^{d_{uv}},
\]
to its HN filtration with respect to the Euler stability condition (see Theorem \ref{thm:HN_filtration} and Definition \ref{def:eulerstab}). The first step in this direction is to establish some general results which hold for a wider class of stability conditions.

\subsection{General Results for Antitone Slopes}\label{sec:genres}

Let $\preceq$ denote the lexicographical total ordering on subintervals of $[0,n-1]$, where $[u,v] \preceq [u',v']$ holds if we have either $u < u'$  or both $u = u'$  and $v \leq v'$. Let $\alpha:Q_0 \to \mathbb{R}$ be a stability condition on $Q$ with associated slope $\phi_\alpha$. We say that our slope $\phi_\alpha$ is {\bf antitone} if the inequality $\phi_\alpha(\I[u,v]) \geq \phi_\alpha(\I[u',v'])$ holds in $\mathbb{R}$ whenever we have $[u,v] \preceq [u',v']$. 

\begin{lmm}\label{lmm:semi-stable_intervals}
If $\phi_\alpha$ is antitone, then every interval module $\I[u,v]$ for $[u,v] \in \barc(V)$ is $\alpha$-semistable.
\end{lmm}
\begin{proof}
Fix an interval $[u,v] \in \barc(V)$, and consider a nonzero submodule $U \subset \I[u,v]$. Using the barcode decomposition of $U$ from Theorem \ref{thm:gabriel_An} followed by Lemma \ref{lmm:seq_and_slope}, we obtain the inequality 
\[
\phi_\alpha(U) \leq \max_{[u',v'] \in \barc(U)}\phi_\alpha(\I[u',v']).
\] Thus, we are required prove that $\phi_\alpha(\I[u',v']) \leq \phi_\alpha(\I[u,v])$ holds whenever $\I[u',v']$ is a submodule of $\I[u,v]$. It is readily checked that this submodule condition holds if and only if $u' \geq u$ and $v' = v$. In particular, we need $u' \geq u$ for dimension reasons, and $v' = v$ follows from the fact that if $v$ exceeds $v'$, then the following diagram of vector spaces fails to commute:
\begin{align}\label{eq:barinc}
\xymatrixcolsep{.6in}
 \xymatrix{
 \K=\I[u',v']_{v'} \ar@{->}[r] \ar@{^{(}->}[d] & 0=\I[u',v']_{v'+1} \ar@{^{(}->}[d]\\
 \K=\I[u,v]_{v'}\ar@{->}[r]_-{\text{id}} & \K = \I[u,v]_{v'+1}
 }
\end{align}
Thus, we have $[u,v] \preceq [u',v']$ and the desired conclusion  follows because $\phi_\alpha$ is antitone.
\end{proof}

For antitone $\phi_\alpha$'s we can now describe HN filtrations directly in terms of barcodes.

\begin{prop}\label{prop:HN-persitence_module}
Assume that the slope $\phi_\alpha$ is antitone,  for each $t \in \mathbb{R}$  the HN $\R$-filtration of $V$ evaluated at $t$ is given by
\begin{align}\label{eqn:formula_HN_interval}
V_{\mathbb{R}}(t) \simeq  \bigoplus_{[u,v]}  \I[u,v]^{d_{uv}},
\end{align}
where the sum is over all $[u,v] \in \barc(V)$ satisfying $\phi_\alpha(\I[u,v]) \geq t$.  The multiplicities $d_{uv}$ are the ones appearing in \eqref{eqn:decompo_An} and the isomorphism can be obtained by restricting the decomposition from \eqref{eqn:decompo_An}.
\end{prop}
\begin{proof} We know from Lemma \ref{lmm:semi-stable_intervals} that interval modules are semistable for antitone $\phi_\alpha$. By Remark \ref{rem:hnss}, the HN filtration of such a module is the trivial one ($0\subsetneq \I[u,v]$), which clearly satisfies \eqref{eqn:formula_HN_interval}. By Proposition \ref{prop:direct_sum_HN}, the left side of \eqref{eqn:formula_HN_interval} is (also) additive, whence $V$ satisfies (\ref{eqn:formula_HN_interval}) as claimed.
\end{proof}

\begin{cor}\label{cor:HN-filtr_An_quotients} Assume that $\phi_\alpha$ is antitone and that the HN filtration $V^\bullet$ of $V$ has length $\ell$. Then for each $j$ in $\set{1,\ldots,\ell}$, we have an isomorphism
\[\frac{V^j}{V^{j-1}} \simeq \bigoplus_{[u,v]} \I[u,v]^{d_{uv}},\]
where the sum ranges over all $[u,v] \in \barc(V)$ satisfying $\phi_\alpha(\I[u,v]) = \phi_\alpha(V^j/V^{j-1})$, and where $d$ indicates the multiplicity function of $V$ from \eqref{eqn:decompo_An}.
\end{cor}

\subsection{Specific Results for Euler Slopes} We continue to work with a fixed persistence module $V$ as in \eqref{eq:pmod}, but now specialise to the Euler stability condition $\epsilon:Q_0 \to \mathbb{R}$ from Definition \ref{def:eulerstab}. It is immediate from this definition that the Euler slope of $V$ is 
\begin{align}\label{eq:euslopeAn}
\phi_\epsilon(V) = \frac{\dim V_0}{\sum_{i=0}^{n-1}\dim V_i}.
\end{align}

Here is the full version of Theorem ({\bf A}) from the Introduction.
\begin{thm}\label{thm:hn_zerobars}
Let $[u,v] \mapsto d_{u,v}$ be the multiplicity function of $V$ as described in \eqref{eqn:decompo_An}, and let $J$ be the (possibly empty) set of all indices $j \in \set{0,\ldots,n-1}$ for which $[0,j]$ lies in $\barc(V)$. Then, 
\begin{enumerate}
    \item the length $\ell$ of the (Euler) HN filtration $V^\bullet$ is \[
    \ell = 1+ (\# J);
    \] and moreover,
    \item if the elements of $J$ are $j_1 < \ldots < j_{\ell-1}$ and $J^*$ is the set of all intervals $[u,v] \in \barc(V)$ with $u \neq 0$, then we have
    \[
    \frac{V^k}{V^{k-1}} = \begin{cases}\I[0,j_k]^{d_{0,j_k}} & \text{if } k < \ell \\
    \bigoplus_{[u,v] \in J^*} \I[u,v]^{d_{u,v}} & \text{if } k=\ell.
    \end{cases} 
    \]
\end{enumerate}
\end{thm}
\begin{proof} Replacing $V$ by an indecomposable module $\I[u,v]$ in the formula \eqref{eq:euslopeAn} gives
\[
\phi_\epsilon(\I[u,v]) = \begin{cases}\frac {1}{v+1}, & \text{if }u=0\\ 0&\text{otherwise.}\end{cases}
\]
Thus, $\phi_\epsilon$ is antitone and we may apply Corollary \ref{cor:HN-filtr_An_quotients} to establish both assertions. 
\end{proof}

The preceding result establishes a direct relationship between barcode decompositions and (Euler) HN filtrations of ordinary persistence modules --- to obtain the HN filtration from the barcode inductively, one lexicographically orders the intervals in $\barc(V)$ with left endpoint zero as $[0,j_0] \prec [0,j_1] \prec \cdots \prec [0,j_{n-1}]$; then for $k < \ell$, we have 
\[
V^{k} \simeq V^{k-1} \oplus \I[0,j_k]^{d_{0,j_k}}.
\]
Conversely, it is possible to extract the multiplicity of any interval of the form $[0,v]$ in $\barc(V)$ by using the HN filtration: there is at most one index $k$ for which the equality $\phi_\epsilon(V^k/V^{k-1}) = 1/(v+1)$ holds. If such a $k$ exists, then we have $d_{0,v} = \dim(V^k/V^{k-1})_{x_0}$; otherwise, $[0,v]$ is not in $\barc(V)$. There is, however, no separation of intervals $[u,v] \in J^*$ from each other since all of them have slope $0$. Two remedies are available for this unfortunate incompleteness --- the first of these comes in the form of slopes which separate all intervals in $\barc(V)$. Any such {\em totally stable} slope \cite{kinser}, if used instead of $\epsilon$ in the proofs above, would allow us to recover the entire barcode of $V$ via successive quotients of the associated HN filtration. The second remedy is described below.

\begin{rmk}
\label{rem:recursive}Let us label the vertices and edges of the underlying quiver $Q$ as follows:
\[
\xymatrixcolsep{.6in}
\xymatrix{
x_0 \ar@{->}[r]^-{e_1} & x_1 \ar@{->}[r]^-{e_2} & \cdots \ar@{->}[r]^-{e_{n-1}} & x_{n-1}. 
}
\]
For each integer $k$ in $\set{0,\ldots,n-1}$, define $Q^{\geq k}$ to be the type $\mathbb{A}_{n-k}$ subquiver of $Q$ consisting of vertices $\set{x_j \mid k \leq j \leq n-1}$ and all the edges between them. Let $V^{\geq k}$ be the representation of $Q^{\geq k}$ obtained by restricting $V$ to  $Q^{\geqslant k}$. By the restriction principle \cite{zigzag_pers}, the decomposition (\ref{eqn:decompo_An}) becomes 
\[
V^{\geqslant k}\simeq\bigoplus_{[u,v]} \I[\max(u-k,0),v-k]^{d_{uv}},
\]
where the sum is over all intervals $[u,v]\in \barc(V)$ satisfying $v \geq k$. For a fixed interval  $[k,v]$ in $\barc(V)$, the above decomposition induces the relation on multiplicities
\[
d_{k,v}=d^{\geq k}_{0,v-k}  - \sum_{[u,v]}d_{u,v}.
\]
Here $d^{\geq k}_{0,v-k}$ denotes the multiplicity of $[0,v-k]$ in  $\barc(V^{\geqslant k})$, and the sum is over intervals $[u,v]$ in $\barc(V)$ satisfying $u < k$. By Theorem \ref{thm:hn_zerobars}, we have that $d^{\geq k}_{0,v-k}$ can be recovered from the HN filtration of $V^{\geq k}$ along the Euler stability condition. Thus, we can inductively reconstruct the entire multiplicity function of $V$ from the HN filtrations of $\set{V^{\geq k} \mid k \geq 0}$.
\end{rmk}

\section{Unwinding Affine Type $\widetilde{\mathbb{A}}$ Quivers}\label{sec:An-tilde-unwind}

We say that a quiver is of {\bf type $\mathbb{A}_\infty$} whenever its underlying graph has the form
\[
\xymatrixcolsep{.6in}
\xymatrix{
\cdots \ar@{-}[r]^-{a_{-1}} & y_{-1} \ar@{-}[r]^-{a_{0}}  & y_0 \ar@{-}[r]^-{a_{1}} & y_1 \ar@{-}[r]^-{a_{2}} & \cdots
}
\]
Explicitly, the vertex set is identified with the set of integers $\Z$, and there is a unique edge between every adjacent pair of integers. Thus, representations of $\mathbb{A}_\infty$ quivers are {\em infinite} zigzag persistence modules. 

We now fix a quiver $Q = (s,t:Q_1 \to Q_0)$ of type $\widetilde{\mathbb{A}}_n$ labelled as in \eqref{eq:Antilde}.

\begin{defi}
The {\bf unwinding} of $Q$ is a quiver $\UQ$ of type $\mathbb{A}_\infty$ whose vertex and edge sets are $\UQ_0 := Q_0 \times \Z$ and $\UQ_1 := Q_1 \times \Z$, with source and target maps $\sigma,\tau:\UQ_1 \to \UQ_0$ given as follows: for all $j \in \set{1,2,\ldots,n-1}$ and $c \in \Z$, we have
\[
\sigma(e_j,c) := (s(e_j),c) \quad \text{ and } \quad \tau(e_j,c) := (t(e_j),c).
\]
If $s(e_0) = x_0$ then the edge $(e_0,c) \in \UQ_1$ has source $(x_0,c)$ and target $(x_{n-1},c-1)$ for each $c \in \Z$. Conversely, if $s(e_0) = x_{n-1}$, then $(e_0,c)$ has source $(x_{n-1},c-1)$ and target $(x_0,c)$.
\end{defi}

Let $\rho: \UQ_0 \to \Z$ be the bijection $(x_i,c) \mapsto i+cn$; for each interval $[p,q] \subset \Z$, let $\UQ^{p,q} \subset \UQ$ denote the {\em truncated} subquiver spanned by all vertices satisfying $\rho(x_j,c) \in [p,q]$ and the edges between them.

\begin{ex}\label{ex:An-tilde-unwinding} Here is $\widetilde{\mathbb{A}}_6$ quiver from  Example \ref{ex:An-tilde-reps} depicted above its unwinding. The labels above and below the vertices of the unwinding are  $\rho(\UQ_0)$ and $\UQ_0$ respectively.

\[\begin{tikzpicture}[
roundnode/.style={circle,fill, inner sep=1pt},scale=1,thick,->]

\foreach \i in {-1,...,6}{
    \node[roundnode] (y\i) at (2*\i,0) {};
       \node (\i laby) at (2*\i,0.25) {\footnotesize $y_{\i}$};
    }
    \foreach \i in {0,...,5}{
        \node(\i labxc) at (2*\i,-0.25) {\footnotesize $(x_{\i},0)$};
        }
        
        \node(-1 labx) at (-2,-0.25) {\footnotesize $(x_{5},-1)$};
        \node(6 labxc) at (12,-0.25) {\footnotesize $(x_{0},1)$};
    \foreach \i in {0,...,4}{
         \node[roundnode] (x\i) at (2*\i,2) {};   
         \node (\i labx) at (2*\i,1.75) {\footnotesize $x_{\i}$};
}
             \node[roundnode] (x5) at (4,3.5) {};  
                    \node (5 lab) at (4,3.25) {\footnotesize $x_{5}$};
    \foreach \i in {0,1,3,4}{
       \draw (x\i)  -- (x\the\numexpr \i+1);
      \draw (y\i) -- (y\the\numexpr\i+1);
}
 \draw (x5)  -- (x\the\numexpr0);
  \draw (y5)  -- (y6);
    \draw (y-1)  -- (y0);
    \draw (y3)  -- (y2);
    \draw (x3)  -- (x2);
    
      \node (d-3) at (-2.5,0) {\dots};
    \node (d13) at (12.5,0) {\dots};
\end{tikzpicture}
\]
\end{ex}

As before, we only consider representations whose assigned vector spaces are all finite-dimensional. Such a representation $W \in \Rep(\UQ)$ is said to be {\em periodic} whenever we have: 
\begin{align*}
W_{(x_j,c)}=W_{(x_j,c+1)} & \text{ for all } (x_j,c) \in \UQ_0, \text{ and }\\
W_{(e_j,c)}=W_{(e_j,c+1)} & \text{ for all } (e_j,c) \in \UQ_1.
\end{align*} 
We denote by $\PRep(\UQ)$ the full subcategory of $\Rep(\UQ)$ which consists of all such periodic representations --- crucially, we do not require that the morphisms in this category are periodic. 

\begin{prop}\label{prop:prep-indec}
Up to isomorphism, the indecomposable representations in $\PRep(\UQ)$ are either 
\begin{enumerate} 
 \item the infinite interval module $\I[-\infty,\infty]$, or 
 \item the direct sum $\mathbf{\Sigma}[u,v] := \bigoplus_{c \in \Z}\I[u+cn,v+cn]$ for some interval $[u,v] \subset \Z$.
 \end{enumerate} 
\end{prop}
\begin{proof}
Since $\PRep(\UQ)$ is a subcategory of $\Rep(\UQ)$, it follows from Theorem \ref{thm:A-inf-indec} that every representation $W$ in $\PRep(\UQ)$ admits an irreducible decomposition in $\Rep(\UQ)$ of the form
\begin{align}\label{eq:Wdecomp}
W \simeq \bigoplus_{[u,v]} \I[u,v]^{d_{u,v}}.
\end{align}
Here the sum ranges over intervals $[u,v] \subset \Z$ lying in the barcode $\barc(W)$, and $d_{u,v}$ is the  multiplicity of each such interval as described in \eqref{eqn:decompo_An}. We first claim that the multiplicity function is $n$-periodic --- namely, for all $[p,q] \subset \Z$, we have $d_{p,q} = d_{p+n,q+n}$. To verify this claim, let $W^{p,q}$ denote the representation of $\UQ^{p,q}$ obtained by restricting $W$ to the available vertices and edges. By \eqref{eq:Wdecomp}, we have
\[
W^{p-1,q+1} \simeq \bigoplus_{[u,v]} \I\big[[u,v]\cap[p-1,  q+1]\big]^{d_{u,v}}
\]
where the sum is over intervals $[u,v] \in \barc(W)$ that have nonempty intersection with $[p-1,q+1]$. The desired $n$-periodicity of $d$ now follows from the fact that the left side remains invariant if $p$ and $q$ are replaced by $p+n$ and $q+n$ respectively. Similarly, $\barc(W)$ can not contain any intervals of the form $[u,\infty]$ or $[-\infty,v]$ --- the inclusion of any $[u,\infty]$ interval would, for instance, force the inclusion of $[u+cn,\infty]$ for all $c \in \Z$, which contradicts the finite dimensionality of $W$. Thus, \eqref{eq:Wdecomp} becomes
\begin{align*}
W &\simeq \I[-\infty,\infty]^{d_{-\infty,\infty}} \oplus \bigoplus_{[u,v]} \left(\bigoplus_{c \in \Z} \I[u+cn,v+cn]\right)^{d_{u,v}}\\
&= \I[-\infty,\infty]^{d_{-\infty,\infty}} \oplus \bigoplus_{[u,v]} \mathbf{\Sigma}[u,v]^{d_{u,v}},
\end{align*}
as desired. Finally, we note that $\I[-\infty,\infty]$ is indecomposable in $\PRep(\UQ)$ because it is indecomposable in the larger category $\Rep(\UQ)$ by Theorem \ref{thm:A-inf-indec}, and that each $\mathbf\Sigma[u,v]$ is indecomposable in $\PRep(\UQ)$ because any nontrivial periodic decomposition of it would produce a nontrivial decomposition of $\I[u,v]$ in a sufficiently long truncation of $\UQ$.
\end{proof}

\begin{rmk}\label{rmk:mult_zigzgag_compute}
The argument given above also establishes that the multiplicity function $\barc ( W)\rightarrow \mathbb N_{>0}$ of $W \in \PRep(\UQ)$  can be recovered from the irreducible decomposition of a sufficiently large but finite zigzag persistence module. More precisely, let $D \in \Z$ be chosen so that for any $[u,v] \in \barc(W)$ there is a $c \in \Z$ satisfying $[u+cn,v+cn] \subset [1,D-1]$. Then there is a one-to-one correspondence between the non-zero summands of the decompositions of $W$ and   the truncation $W^{0,D}$. In fact, $D:=(\dim W_0+2)n$ always suffices: if $[u,v]$ lies in $\barc(W)$, then we have
\[
\#([u,v] \cap n\Z) \leq \dim W_0 \quad \text{ and } \quad v-u < (\dim W_0 + 1)n=D-n,
\] so the translation  of $[u,v]$ by $n\Z$ whose starting point is in $[1,n]$ is included in $[1,D-1]$.
\end{rmk}

There is an evident {\bf lift} map that sends objects of $\Rep(Q)$ to those of $\Rep(\UQ)$ defined as follows. The lift $\cL V$ of a representation $V$ of $Q$ assigns $V_{x_j}$ to every vertex in $\UQ_0$ of the form $(x_j,c)$ and similarly $V_{e_j}$ to every edge of the form $(e_j,c)$ in $\UQ_1$. We note that $\cL V$ is called the {\em infinite cyclic covering} in \cite{burghelea}. 

\begin{ex} Continuing Examples \ref{ex:An-tilde-reps} and \ref{ex:An-tilde-unwinding}, here are the representation $\mathbf N[1,9]$ of the $\widetilde{\mathbb{A}}_6$ quiver from example  \ref{ex:An-tilde-reps} and its lift $\cL (\mathbf N[1,9])$. As before, unlabelled edges carry identity maps:
\[
\xymatrixcolsep{.3in}
\xymatrixrowsep{.25in}
\xymatrix{
&&&& & \mathbb{K} \ar@{->}[dll]& & \\
&&&\mathbb{K} \ar@{->}[r]_{\left[\begin{smallmatrix}0\\1\end{smallmatrix}\right]}& \mathbb{K}^2  \ar@{->}[r] & \mathbb{K}^2   & \mathbb{K}^2  \ar@{->}[r]_-{[1 ~ 0]} \ar@{->}[l] & \mathbb{K} \ar@{->}[ull]\\ \\
\cdots\ar@{->}[r]_-{[1 ~ 0]}&\mathbb K\ar@{->}[r]&\ar@{->}[r]\mathbb K&\mathbb{K} \ar@{->}[r]_{\left[\begin{smallmatrix}0\\1\end{smallmatrix}\right]}& \mathbb{K}^2  \ar@{->}[r] & \mathbb{K}^2   & \mathbb{K}^2  \ar@{->}[r]_-{[1 ~ 0]} \ar@{->}[l] & \mathbb{K} \ar@{->}[r]&\mathbb{K} \ar@{->}[r]&\mathbb{K} \ar@{->}[r]_{\left[\begin{smallmatrix}0\\1\end{smallmatrix}\right]}&\cdots
}
\]
\end{ex}

Given a morphism $f:V \to V'$ in $\Rep(Q)$, the lifted morphism $\cL f: \cL V \to \cL V'$ is given by 
\[
\cL f_{(x_j,c)} = f_{x_j} \text{ for all }c \in \Z.
\]
Thus, the morphisms of $\PRep(\UQ)$ lying in the image of $\cL$ are always periodic even though -- as mentioned above -- morphisms in $\PRep(Q)$ are not periodic in general. As a consequence, we can not expect  indecomposable representations of the form $\N[u,v]$ and $\T[\lambda;w]$ from Definition \ref{def:indAtilde} to lift to indecomposable representations in $\PRep(\UQ)$ from Proposition \ref{prop:prep-indec}. However, we have the following convenient result.

\begin{prop}\label{prop:image_indec_on_zigzag} Let $\cL:\Rep(Q) \to \PRep(\UQ)$ be the lift functor. Then,
\begin{enumerate} 
\item for every interval $[u,v] \subset \Z$, we have
\[
\cL(\N[u,v]) = \mathbf{\Sigma}[u,v],
\]
where the right side is as defined in Proposition \ref{prop:prep-indec}. And moreover,
\item for each  field element $\lambda \neq 0$ in $\K$ and integer $w \geq 1$, we have
\[
\cL(\T[\lambda;w]) = \I[-\infty,\infty]^w.
\]
\end{enumerate}
\end{prop}

\begin{proof}
The matrices in the description of $\N[u,v]$ from Definition \ref{def:indAtilde} are already in barcode form (see \cite[Definition 2.2]{jacquard2022space}), so the desired interval decomposition of ${\cL}\N[u,v]$ can be directly inferred from these matrices. Moreover, since all the matrices in the description of $\T[\lambda;w]$ are isomorphisms, it follows from Theorem \ref{thm:A-inf-indec} that there exists a change of basis which turns them all into identities. 
\end{proof}

\section{HN Filtrations of Affine Type $\widetilde{\mathbb{A}}$ Quivers}\label{sec:HN_of_An_tilde}

Let $Q$ be a quiver of type $\widetilde{\mathbb{A}}_n$ as defined in \eqref{eq:Antilde}, and fix a representation $V \neq 0$ of $Q$ valued in vector spaces over an algebraically closed field $\K$. Our goal here is to relate the HN filtration of $V$ along the Euler stability condition (from Definition \ref{def:eulerstab}) to the barcode decomposition of the lifted representation $\cL V \in \Rep(\UQ)$ which has been described in Proposition \ref{prop:prep-indec}. 

The first task at hand is to compute the Euler slopes of all indecomposables in $\Rep(Q)$ --- we recall from Theorem \ref{thm:an-tilde-indec} that these have form $\N[u,v]$ and $\T[\lambda;w]$, both of which have been described in Definition \ref{def:indAtilde}. There is a single number $p_{u,v} \in \set{0,1,2}$ which determines the sign of the Euler slope of each $\N[u,v]$, and it admits the following graphical description (with $u' := u \text{ mod } n$ and $v' := v \text{ mod } n$):
  \[
\begin{tikzpicture}[
roundnode/.style={circle,draw, inner sep=1.5pt},scale=0.75,thick,->]
  \node[roundnode] (0) at (0,0) {};
          \node[roundnode,label=above:$x_{u'}$](1) at (1,0) {};
    \node[roundnode] (2) at (2,0) {};
            \node[roundnode](3) at (3.5,0) {};
        \node[roundnode,label=above:$x_{v'}$](4) at (4.5,0) {};
          \node[roundnode](5) at (5.5,0) {};
                    \node (lab) at (2.75,-1){$p_{u,v}=0$};
    \draw[->] (1) to (0);
       \draw[-] (1) to (2);
          \draw[opacity=0] (2) to node[opacity=1]{$\dots$} (3);
          \draw[-] (3) to (4);
          \draw[->] (4) to (5);
    \end{tikzpicture}\qquad
    \begin{tikzpicture}[
roundnode/.style={circle,draw, inner sep=1.5pt},scale=0.75,thick,->]
  \node[roundnode] (0) at (0,0.25) {};
          \node[roundnode,label=above:$x_{u'}$](1) at (1,0.25) {};
    \node[roundnode] (2) at (2,0.25) {};
            \node[roundnode](3) at (3.5,0.25) {};
        \node[roundnode,label=above:$x_{v'}$](4) at (4.5,0.25) {};
          \node[roundnode](5) at (5.5,0.25) {};
            \node[roundnode] (0') at (0,-0.25) {};
          \node[roundnode](1') at (1,-0.25) {};
    \node[roundnode] (2') at (2,-0.250) {};
            \node[roundnode](3') at (3.5,-0.25) {};
        \node[roundnode](4') at (4.5,-0.25) {};
          \node[roundnode](5') at (5.5,-0.25) {};
                    \node (lab) at (2.75,-1){$p_{u,v}=1$};
    \draw[->] (0) to (1);
       \draw[-] (1) to (2);
          \draw[opacity=0] (2) to node[opacity=1]{$\dots$} (3);
          \draw[-] (3) to (4);
          \draw[->] (4) to (5);
              \draw[->] (1') to (0');
       \draw[-] (2') to (1');
          \draw[opacity=0] (2') to node[opacity=1]{$\dots$} (3');
          \draw[-] (3') to (4');
          \draw[->] (5') to (4');
    \end{tikzpicture}\qquad 
    \begin{tikzpicture}[
roundnode/.style={circle,draw, inner sep=1.5pt},scale=0.75,thick,->]
  \node[roundnode] (0) at (0,0) {};
          \node[roundnode,label=above:$u'$](1) at (1,0) {};
    \node[roundnode] (2) at (2,0) {};
            \node[roundnode](3) at (3.5,0) {};
        \node[roundnode,label=above:$v'$](4) at (4.5,0) {};
          \node[roundnode](5) at (5.5,0) {};
          \node (lab) at (2.75,-1){$p_{u,v}=2$};
    \draw[->] (0) to (1);
       \draw[-] (1) to (2);
          \draw[opacity=0] (2) to node[opacity=1]{$\dots$} (3);
          \draw[-] (3) to (4);
          \draw[->] (5) to (4);
    \end{tikzpicture}
\]
Explicitly, $p_{u,v}$ counts whether the edge $e_{u'}$ has target $u'$ and whether the edge $e_{v'+1}$ has target $v'$. The cases $p_{u,v} = 0, 1, 2$ are equivalent \cite[Sec 11.3]{gabriel1992representations} to $\N[u,v]$ being called  respectively pre-injective, regular and pre-projective in standard texts \cite{kirillov_quiver, schiffler2014quiver}. 

\begin{prop}\label{prop:An-tilde-euler}
The Euler slopes of indecomposable representations in $\Rep(Q)$ are given as follows.
\begin{enumerate}
    \item For each $[u,v] \subset \Z$, we have
    \[
    \phi_\epsilon(\N[u,v]) = \frac{1-p_{u,v}}{v-u+1}.
    \]
    \item For each nonzero $\lambda \in \K$ and integer $w \geq 1$, we have
    \[
    \phi_\epsilon(\T[\lambda;w]) = 0.
    \]
\end{enumerate}
\end{prop}
\begin{proof}
The second assertion follows from the following elementary calculation: for each $\lambda$ and $w$ as above, by Definitions \ref{def:indAtilde} and \ref{def:eulerstab} we have
\[
\phi_\epsilon(\T[\lambda;w]) = \frac{\sum_{x \in Q_0} (1-\deg_\text{in}(x)) \cdot w}{\sum_{x \in Q_0} w}. 
\]
Since $\sum_{x \in Q_0} \deg_\text{in}(x)$ equals the cardinality of the edge set, we obtain 
\[
\phi_\epsilon(\T[\lambda;w]) = \frac{\#Q_0 - \#Q_1}{\#Q_0}.
\]
The numerator on the right side vanishes, as desired, because both cardinalities equal $n$. 

Turning now to $\N[u,v]$, by the discussion preceding Definition \ref{def:eulerstab} we have
\begin{align} \label{eq:Nslope}
\phi_\epsilon(\N[u,v]) = \frac{\chi(X;\bS\N[u,v])}{\sum_{x \in Q_0} \dim \N[u,v]_x}.
\end{align}
Here $X$ denotes the underlying graph of $Q$, while $\bS:\Rep(Q) \to \textbf{Shf}(X)$ is the associated sheaf functor from Definition \ref{def:repsheaf}. This associated sheaf is locally constant on a decomposition of $X$ into two disjoint intervals $X = A \sqcup B$, whose exact nature depends on the value of $p_{u,v}$. In particular, setting $\ell := \lfloor\frac{v-u}{n}\rfloor$, let $A \subset X$ be the union of cells on which $\N[u,v]$ has dimension $\ell+1$, and let $B$ be the complement of $A$ in $X$ (on which $\N[u,v]$ necessarily has dimension $\ell$). Whether $A$ is closed, clopen or open in $X$ depends on $p_{u,v}$ being $0, 1$ or $2$. Therefore, $A$'s Euler characteristic is
\[
\chi(A) = \begin{cases}
            +1 & p_{u,v} = 0, \\
        \hspace{.13in}0 & p_{u,v} = 1, \\
            -1 & p_{u,v} = 2.
          \end{cases}
\]
Equivalently, we have $\chi(A) = 1-p_{u,v}$. 

Now the numerator in \eqref{eq:Nslope} is given by
\[
\chi(X;\bS\N[u,v]) = \chi(A;\underline{\K^{\ell+1}}_A) + \chi(B;\underline{\K^\ell}_B),
\]
where $\underline{\K^{\ell+1}}_A$ is the constant $\K^{\ell+1}$-valued sheaf on $A$, etc. Thus,
\[
\chi(X;\bS\N[u,v]) = \chi(A) \cdot (\ell+1) + \chi(B) \cdot \ell.
\]
By additivity of the Euler characteristic, we have $\chi(B) = -\chi(A)$ since $X$ is homeomorphic to a circle and satisfies $\chi(X) = 0$. Therefore, $\chi(X;\bS\N[u,v])$ equals $\chi(A)$, which in turn equals $1-p_{u,v}$. Finally, the denominator in \eqref{eq:Nslope} is 
\begin{align*}
& (\ell+1) \cdot (v'-u'+1) + \ell \cdot (n-(v'-u'+1)) \\
=& (v'-u'+1)+\ell n,
\end{align*} which, by definition of $\ell$, simplifies to $v-u+1$ as desired.
\end{proof}

Next, we seek to establish that all the indecomposable representations in $\Rep(Q)$ are $\epsilon$-semistable. In light of Lemma \ref{lmm:semi-stable_intervals}, it suffices to show that the inequality $\phi_\epsilon(I) \leq \phi_\epsilon(I')$ holds whenever we have an injective morphism $I \hookrightarrow I'$ between a pair of indecomposables in $\Rep(Q)$. There are only three nontrivial cases of $I \hookrightarrow I'$ to consider, which -- based on the slope computations from Proposition \ref{prop:An-tilde-euler} -- we call $(++)$, $(0-)$ and $(--)$. Explicitly,
\begin{itemize}
    \item[$(++)$] has $\N[u,v] \hookrightarrow \N[\bu,\bv]$ with $p_{u,v} = 0 = p_{\bu,\bv}$,
    \item[$(0-)$] has $\T[\lambda;w] \hookrightarrow \N[u,v]$ with $p_{u,v} = 2$,
    \item[$(--)$] has $\N[u,v] \hookrightarrow \N[\bu,\bv]$ with $p_{u,v} = 2 = p_{\bu,\bv}$.
\end{itemize}

\begin{rmk}\label{rem:injhed}
If we have an injective morphism $U \hookrightarrow V$ in $\Rep(Q)$, then for every edge $e \in Q_1$ the linear map $U_e$ is precisely the restriction of $V_e$. It follows that if $V_e$ is injective, then so is $U_e$. We note that all the edges are assigned injective maps in $\T[\lambda;w]$; but exactly $2-p_{u,v}$ edges are assigned a map with a nontrivial kernel by $\N[u,v]$. Therefore, the case $(+0)$ involving $\N[u,v] \hookrightarrow \T[\lambda;w]$ for $p_{u,v} = 0$ never arises. And similarly, all the cases involving $\N[u,v] \hookrightarrow \N[\bu,\bv]$ with $p_{u,v} < p_{\bu,\bv}$ are impossible.
\end{rmk}

The following proposition can be seen as a corollary of \cite[Prop 5.2]{stable_quiver}. We prove it here without using any property of tame hereditary algebras except the decomposition from definition \ref{def:indAtilde}.

 \begin{prop}\label{prop:ss_An_tilde}
 The indecomposable representations of $Q$ are all $\epsilon$-semistable.
 \end{prop}
 \begin{proof}
 We treat each of the three cases separately.
 
 \medskip
 
 {\bf Case} $(++)$: The injectivity requirement from Remark \ref{rem:injhed} forces both $u=\bu\text{ mod }n$ and $v = \bv \text{ mod }n$, and (translating if necessary) we may as well assume that $\bu=u$ and $\bv=v+cn$ for some $c \in \Z$. If $c$ is negative, then $(v-u-1)$ exceeds $(\bv-u-1)$, which implies $\phi_\epsilon(\N[u,v]) < \phi_\epsilon(\N[u,\bv])$ by Proposition \ref{prop:An-tilde-euler}. Thus, we assume that $c$ is non-negative and seek to verify that $c=0$. Consider the lifts 
 \[
 W := \cL(\N[u,v]) \quad \text{and} \quad W' := \cL(\N[u,\bv])
 \] which are defined on the unwinding $\UQ$. By Proposition \ref{prop:image_indec_on_zigzag}, we have $W = {\bf \Sigma}[u,v]$ and $W' = {\bf \Sigma}[u,\bv]$. The injective map $\N[u,v] \hookrightarrow \N[u,\bv]$ in $\Rep(Q)$ lifts to an injective map $\iota:W \hookrightarrow W'$ in $\PRep(\UQ)$. Restricting to the {\em finite} type $\mathbb{A}$ subquiver  obtained by truncating $\UQ$ to $[\rho^{-1}(u)-1,\rho^{-1}(\bv)+1]$, we obtain a distinguished interval $[u,v]$ in $\barc(W)$ that is generated by the kernel of $W_{\rho^{-1}(u)}$. Similarly, there is an interval $[u,\bv]$ in $\barc(W')$ generated by the kernel of $W'_{\rho^{-1}(u)}$. The fact that $\iota$ induces an injective map $\I[u,v] \hookrightarrow \I[u,\bv]$ in $\Rep(Q')$ now forces $v = \bv$ by \eqref{eq:barinc}. Thus, $\N[u,v]$ and $\N[\bu,\bv]$ have equal slopes by Proposition \ref{prop:An-tilde-euler}.
 
 \medskip
 
 {\bf Case} $(0-)$: Proceeding as in the previous case, we define the relevant lifts $W$ of $\T[\lambda;w]$ and $W'$ of $\N[u,v]$. By Proposition \ref{prop:image_indec_on_zigzag}, we have
 \[
 W = \I[-\infty,\infty]^w \quad \text{and} \quad W' = {\bf \Sigma}[u,v].
 \]
 We now claim that there is in fact no injective map $\T[\lambda;w] \hookrightarrow \N[u,v]$ in $\Rep(Q)$; if such a map existed, it would induce an injective map $W \hookrightarrow W'$ in $\PRep(\UQ)$. But by the same reasoning as in \eqref{eq:barinc}, no infinite interval can map injectively to a direct sum of finite intervals.
 
 \medskip
 
 {\bf Case} $(--)$: Since $\N[u,v]$ is a sub-representation of $\N[\bu,\bv]$ by assumption, we have the inequality $\dim \N[u,v]_x \leq \dim \N[\bu,\bv]_x$ for every vertex $x \in Q_0$. By Definition \ref{def:indAtilde}, these inequalities force $(\bv-\bu+1) \leq (v-u+1)$, whence Proposition \ref{prop:An-tilde-euler} guarantees that we have $\phi_\epsilon(\N[u,v]) \leq \phi_\epsilon(\N[\bu,\bv])$ as desired.
 \end{proof}
 
Here is our main result, which is Theorem ({\bf B}) from the Introduction.

\begin{thm}\label{thm:main}
Let $Q$ be an acyclic quiver of type $\widetilde{\mathbb{A}}_n$ and let $V$ be a representation of $Q$ over an algebraically closed field. Let $d^{*}:\barc(\cL V)\rightarrow \mathbb N_{>0}$ be the multiplicity function of the lift of $V$, and let $\eta(V):=\left(\udim_{V^j/V^{j-1}}\right)_{j}$ be the dimension vectors of successive quotients in the HN filtration of $V$ along $\epsilon$. Then,
\begin{enumerate}[label=(\roman*)]
    \item $d^{*}$ and $\eta(V)$ can  be computed in time $\mathcal O(\Delta^4 n)$ where $\Delta := \max_{x\in Q_0} \dim V_x$.
    \item For $[u,v]\subset \Z$, the multiplicity of $\mathbf N[u,v]$ and its appearance in the irreducible decomposition of $V$ can be retrieved from $d^{*}$. 
    \item For $[u,v]\subset \Z$ satisfying $p_{u,v}\neq 1$, the multiplicity of $\mathbf N[u,v]$ in the irreducible decomposition of $V$ can be retrieved from $\eta(V)$. 
\end{enumerate}
\end{thm}

\begin{proof}
From Theorem \ref{thm:an-tilde-indec}, we know that there exist integers $d_{u,v} > 0$ and $d_{\lambda;w} > 0$ so that $V$ decomposes as
\[
V \simeq \bigoplus_{[u,v]} \N[u,v]^{d_{u,v}} \oplus \bigoplus_{\lambda,w} \T[\lambda;w],
\] where the first sum is over $[u,v]\subset \Z$ with $u\in [0,n-1]$ such that $\mathbf N[u,v]\in \mathrm{Ind}(V)$ and the second sum is over $(\lambda,w)\in (\mathbb K-\{0\})\times \mathbb N_{>0}$ such that $\mathbf T[\lambda;w]\in \mathrm{Ind}(V)$. Similarly, by Proposition \ref{prop:prep-indec} we have integers $d^*_{u,v} > 0$ and $d^*_\infty > 0$ satisfying
 \[
 {\cL}V \simeq \bigoplus_{[u,v]}\mathbf \Sigma[u,v]^{d^*_{u,v}} \oplus\bigoplus_{(\lambda,w)} \mathbf{I}[-\infty,\infty]^{d^*_\infty}\] 
 By Proposition \ref{prop:image_indec_on_zigzag} and the additivity of the lift functor, we obtain the following relations between the multiplicities of indecomposables in $V$ and in $\cL V$:
   \begin{equation}\left\{\label{eq:rel_mult_V_LV} \begin{array}{ll}
    d_{u,v}&=d^{*}_{u,v} \\
    \sum_{\lambda,w}w \cdot  d_{\lambda;w} &= d^{*}_{\infty}
    \end{array}\right.\end{equation}
    where $[u,v]\subset \Z$ and $(\lambda,w)\in (\mathbb K-\{0\})\times \mathbb N_{>0}$. We use the convention  $d_I=0$ when the indecomposable $I$ is not in $\mathrm{Ind}(V)$, and similarly $d^*_I = 0$  when $I$ is not in $\mathrm{Ind}(\cL V)$. Consider a successive quotient $S^j := V^j/V^{j-1}$ of the HN filtration of $V$. By combining \eqref{eq:rel_mult_V_LV} with Propositions \ref{prop:An-tilde-euler} and \ref{prop:ss_An_tilde}, the dimension vector of $S^j$ can be decomposed as a sum over indecomposables $\mathbf N[u,v]$ with the same slope as $S^j$:
\begin{equation}
    \label{eq:HN_inv_case_slope_not0}\udim_{S^j} = \sum_{\mathbf N[u,v]} d^{*}_{u,v} \cdot \udim_{\mathbf N[u,v]}\end{equation}
except when $\phi_\epsilon(S^j)= 0$, in which case there is an extra term induced by the summands of the form $\mathbf T[\lambda;w]$:
\begin{equation}
    \label{eq:HN_inv_case_slope_0}\udim_{S^j}= 
d_{\infty}^{*} \cdot \udim_{ \mathbf T[1;1]}+\sum_{\mathbf N[u,v]} d^{*}_{u,v} \cdot  \udim_{\mathbf N[u,v]}.\end{equation}
This formula simplifies further since $\udim_{\T[1;1]}$ equals one at every vertex of $Q$. We now turn to the three assertions which must be established.

\begin{enumerate}[label=(\roman*)]
\item From Remark \ref{rmk:mult_zigzgag_compute} we know that $d^*$ can be deduced from the irreducible decomposition  of a  zigzag persistence module of length $\mathcal{O}(\Delta n)$ with all spaces of dimension at most $\Delta$. This gives the desired worst-case complexity as the algorithm \cite{zigzag_pers} involves performing one Gaussian elimination per map of the zigzag persistence module. Moreover,  equations \eqref{eq:HN_inv_case_slope_not0} and \eqref{eq:HN_inv_case_slope_0} show that $\eta(V)$ can be deduced from $d^{*}$ by updating at most $\Delta$ times one of  the dimensions  $\dim S^j_x$ at some fixed vertex $x$.
    \item    This is a direct consequence of the relations \eqref{eq:rel_mult_V_LV}.
    \item Consider an interval $[u,v]\subset\Z$ such that $p_{u,v}\neq 1$, and let  $\mathcal A_{u,v}$ be the set of indecomposables with the same Euler slope as $\mathbf N[u,v]$ that  occur in the irreducible decomposition of $V$. By Proposition \ref{prop:An-tilde-euler},  the slope of $\mathbf N[u,v]$ is nonzero and the  intervals $[u',v']\subset \Z$ satisfying $\mathbf N[u',v']\in \mathcal A_{u,v}$  have the same value of $p$ and the same length as $[u,v]$.
    
    We claim that the membership of $\mathbf N[u,v]$ in $\mathcal A_{u,v}$  and its multiplicity in the decomposition of $V$ can be obtained via the relation
    \[\dim S^j_{u_*}-\dim S^j_{(u-1)_*} = \begin{cases}d_{u,v}&\text{if }\mathbf N[u,v]\in\mathcal A_{u,v}\\0&\text{otherwise}\end{cases}\]
where $a_*$  denotes the residue modulo $n$ of an integer $a$. Indeed, by equation \eqref{eq:HN_inv_case_slope_not0}, the difference $\dim S^j_{u_*} - \dim S^j_{(u-1)_*}$ can be  decomposed as
 \[
 \sum_{\mathbf N[u',v']\in \mathcal A_{u,v}} d_{u',v'} \cdot \left[\#([u',v']\cap (n\Z+u))-\#([u',v']\cap (n\Z+u-1))\right].
 \] 
The summands in the above sum are null unless either $u'=u\text{ mod  }n$ or $v'+1 =u \text{ mod }n$. The latter case does not happen as $p_{u',v'}=p_{u,v}$ forces $e_{u\text{ mod }n}$ and $e_{(v'+1) \text{ mod }n}$ to have  different orientations. The former case, $u'=u\text{ mod }n$, only happens when $[u',v']$ is a translation by $n\Z$ of $[u,v]$ since $[u',v']$ and $[u,v]$ are required to have the same length.
\end{enumerate}
\end{proof}
 
Finally, we note that although the results obtained above require $\K$ to be be algebraically closed, one can always fix bases for the vector spaces of a given $V \in \Rep(Q)$ and consider a new representation $V^+$ with the same matrix description but now over the algebraic closure $\overline{\K}$ of $\K$. The isomorphism class of $V^+$ only depends on the isomorphism class of $V$, and as such, $\eta(V^+)$ still constitutes an (incomplete) invariant of $V$. By Theorem \ref{thm:main} above, this invariant can be related to the barcode of $\cL V^+$ in the category of $\UQ$-representations valued in vector spaces over $\overline{\K}$.

\bibliographystyle{abbrv}
\bibliography{ref}

\end{document}